\documentclass[10pt,a4paper]{article}
\usepackage{CJKutf8}
\usepackage{latexsym}
\usepackage{amsmath}
\usepackage{amssymb}
\usepackage{amsthm}
\usepackage{amscd,color}

\newtheorem{theorem}{Theorem}[section]
\newtheorem{lemma}[theorem]{Lemma}

\theoremstyle{definition}

\newtheorem{remark}[theorem]{Remark}

\theoremstyle{remark}

\begin{document} 
\begin{CJK}{UTF8}{min}

\title{Fractional Powers of Operators: Characterization via $\delta$-Regularized Logarithmic Representation}
\author{
Yoritaka Iwata\\
Osaka University of Economics and Law, \\
Gakuonji 6-10, Yao, Osaka 581-0853, Japan. 
\thanks{
E-mail address: 
{\bf y-iwata@s.keiho-u.ac.jp}}}

\date{}
\maketitle


\begin{abstract}
For the fractional powers of operators, the classical formulation by V. Balakrishnan inherently requires positivity or specific sectorial conditions on the generator (i.e., the generation of analytic semigroups).
To overcome these limitations, this paper presents a novel approach to constructing fractional powers of operators $D^k$ through the $\delta$-regularized logarithmic representation of infinitesimal generators $D$ within the framework of $C^0$-semigroup theory on Banach spaces.
The proposed method bypasses the conventional geometric constraints by employing an algebraic $\delta$-regularization for bounded operators.
Specifically, by applying a complex-analytic logarithmic representation to a family of bounded operators derived from the resolvent of the semigroup, the fractional powers for generators of general $C^0$-semigroups are proved to be uniquely and rigorously well-defined under the minimal algebraic assumption of invertibility, and completely independent of the choice of the regularization parameter $\delta > 0$.
The theoretical framework established in this study provides extensive potential for applications, including regularity estimates for non-analytic semigroups and non-autonomous systems where the infinitesimal generators depend explicitly on the time variable.
\end{abstract}

\vspace{2mm}
\vspace{2mm}

\section{Introduction}
Fractional powers of operators provide an indispensable tool for analyzing complex phenomena that escape the grasp of conventional integer-order powers, as well as for discussing the regularity of solutions to abstract evolution equations. 
Within the framework of operator semigroup theory, the origin of fractional powers can be traced back to the formulation introduced by V. Balakrishnan~\cite{60Balakrishnan} and systematized by K. Yosida~\cite{65yosida}, given by the following expression:
\begin{equation} \label{balak}
(-A)^{\alpha} u = \frac{1}{\Gamma(-\alpha)} \int_0^\infty t^{-\alpha-1} (U(t)-I) u \, dt \quad (0 < \alpha < 1),
\end{equation}
where $\alpha$ is a real number (specifically $0 < \alpha < 1$), and $u$ belongs to the domain $D(A) \subset X$. 
Here, $A$ is the infinitesimal generator of a strongly continuous semigroup $\{ U(t) \}_{t \ge 0}$ on a Banach space $X$ (in the following, $C^0$-semigroup), which provides the solution representation for the abstract evolution equation $du/dt = Au$.
However, for Eq.~\eqref{balak} to hold, it is an essential assumption that the operator $-A$ satisfies a non-negativity (or sectorial) condition; that is, its resolvent set contains the open right half-plane and satisfies the condition
\begin{equation}
\sup_{\operatorname{Re}(\lambda)>0} | \operatorname{Re} (\lambda)| \cdot
\|(\lambda I - A)^{-1} \| < \infty.
\end{equation}
This condition implies that the spectrum of $-A$ is contained within a sectorial region in the complex plane with its vertex at the origin (the sectorial condition). 
This requirement has been the primary reason why conventional theories have been inevitably restricted to the framework of analytic semigroups.

Tracing the theoretical background, the foundational idea of Eq.~\eqref{balak} inherently arises from Cauchy's integral formula for derivatives in complex analysis:
\begin{equation}
\frac{d^n f(z_0)}{dz^n} = \frac{n!}{2 \pi i} \int_{\Gamma} \frac{f(z)}{(z-z_0)^{n+1}} dz.
\end{equation}
By generalizing the discrete factorial $n!$ to the Euler gamma function $\Gamma(n+1)$, and extending the integer $n$ to a real number $\alpha$, a theoretical path was opened to define fractional powers. 
To transpose this complex-analytic concept into the operator-theoretic setting on Banach spaces, we employ the Riesz--Dunford functional calculus, expressing $(-A)^{-\alpha}$ via a contour integral involving the resolvent $(zI - A)^{-1}$. 
By substituting the Laplace transform representation of the resolvent, $(zI - A)^{-1} = \int_0^\infty e^{-zt} U(t) \, dt$, into the Dunford integral and exchanging the order of integration, the complex contour integral is successfully converted into a real-time integral over $t \in (0, \infty)$ for the fractional integral $(-A)^{-\alpha}$ with $0 < \alpha < 1$:
\begin{equation} \label{frac_int}
(-A)^{-\alpha} u = \frac{1}{\Gamma(\alpha)} \int_0^\infty t^{\alpha-1} U(t) u \, dt.
\end{equation}

To transition from the fractional integral $(-A)^{-\alpha}$ to the positive fractional power $(-A)^\alpha$, the operator is naturally decomposed into the generator $(-A)$ and a $(1-\alpha)$-th order integration:
\begin{equation} \label{decomp}
(-A)^\alpha = (-A) \cdot (-A)^{-(1-\alpha)}.
\end{equation}
Substituting Eq.~\eqref{frac_int} with $1-\alpha$ into Eq.~\eqref{decomp} yields a kernel factor $t^{(1-\alpha)-1} = t^{-\alpha}$. 
By applying the generator $(-A)$ to the semigroup via $-A U(t) u = -\frac{d}{dt} U(t) u$, integration by parts transfers the action of $(-A)$ onto the kernel, introducing an additional factor of $t^{-1}$ that yields $t^{-\alpha-1}$. 
Simultaneously, evaluating the boundary terms yields the difference operator $U(t) - I$, which effectively regularizes the behavior at $t = 0$. 
Finally, utilizing the functional equation of the gamma function, $\Gamma(1-\alpha) = -\alpha \Gamma(-\alpha)$, Balakrishnan successfully derived the formulation in Eq.~\eqref{balak}, where $\Gamma(-\alpha)$ and $t^{-\alpha-1}(U(t)-I)$ naturally appear.
In this manner, Balakrishnan's formulation beautifully unified the analytic continuation via the gamma function with the semigroup theory connected to the resolvent via the Laplace transform.

Despite its mathematical elegance, Balakrishnan's classical formulation encompasses inherent theoretical and practical drawbacks. 
First, the rigid requirement of sectoriality severely restricts its domain of applicability, leaving general non-analytic $C^0$-semigroups beyond its reach. 
Second, handling the domains of unbounded operators in fractional powers introduces considerable complexity when extending the theory to non-autonomous evolution equations—a challenge that necessitated the elaborate analytical frameworks of Kato~\cite{61kato,80kato} and Tanabe~\cite{79tanabe}. 
Furthermore, formulations relying on improper integral representations pose non-trivial obstacles for numerical approximations and the algebraic treatment of operator products, such as the Baker--Campbell--Hausdorff formula.

To circumvent these limitations, the construction of fractional powers via logarithmic operators was investigated by Okazawa~\cite{86okazawa, 96okazawa} for sectorial and accretive operators. 
However, these approaches still fundamentally rested upon geometric sectorial conditions on the spectrum, leaving the extension to general $C^0$-semigroups an open challenge.

To bridge this gap, this paper utilizes the operator logarithmic representation framework introduced in~\cite{23iwata}. 
The key strategy of our approach is to map the inherently unbounded generator $D$ into a family of bounded operators constructed via the resolvent of the semigroup. 
Through an algebraic $\delta$-regularization of these bounded operators, the conventional geometric constraints on the spectrum can be entirely bypassed. 
The main objective of this paper is to rigorously construct the fractional power $D^k$ for generators of general $C^0$-semigroups via this $\delta$-regularized logarithmic representation, and to prove that the resulting operator is uniquely and well-defined, completely independent of the choice of the regularization parameter $\delta > 0$.

\section{Mathematical Settings}
Let $X$ be a Banach space equipped with norm $\|\cdot\|_X$. 
As an auxiliary equation, let us consider a first-order homogeneous abstract evolution equation:
\begin{equation} \label{mastereq}
\begin{cases}
\dfrac{du}{d\lambda} - D u = 0, & \lambda \in (a,b], \\[2.5mm]
u(a) = u_a \in X,
\end{cases}
\end{equation}
in $X$, where $\lambda$ denotes the evolution parameter introduced specifically to facilitate the construction of fractional powers of operators.
The operator $D : D(D) \subset X \to X$ is an infinitesimal generator on $X$ that is explicitly independent of $\lambda$.
In the present study, while $D$ is assumed to generate a $C^0$-semigroup, it is not necessarily restricted to generating an analytic semigroup. 
An equation of the form~\eqref{mastereq} under this general setup is generally classified as a hyperbolic abstract evolution equation~\cite{70kato, 73kato}.
Let $\{ U(\lambda) \}_{\lambda \ge 0}$ denote the $C^0$-semigroup of operators on $X$ generated by $D$, so that $U(\lambda) \in B(X)$ for each $\lambda \ge 0$, where $B(X)$ denotes the space of bounded linear operators on $X$. 
Utilizing this evolution operator, the unique solution to Eq.~\eqref{mastereq} is explicitly represented as $u(\lambda) = U(\lambda - a) u_a$.

Since $U(\lambda) \in B(X)$, it has a non-empty resolvent set. 
For this bounded operator $U(\lambda)$, let us introduce a family of operators $I_{\eta}(\lambda)$, which plays a fundamental role in constructing the logarithmic representation of the generator $D$ discussed in the subsequent sections:
\begin{equation}
I_{\eta}(\lambda) = (I - \eta^{-1}U(\lambda))^{-1},
\end{equation}
where $I$ denotes the identity operator on $X$, and $\eta \in \mathbb{C}$ is a complex parameter chosen appropriately (e.g., $|\eta| > r(U(\lambda))$, where $r(\cdot)$ is the spectral radius) such that  $I_{\eta}(\lambda) \in B(X)$ exists and its operator logarithm can be rigorously defined. 
In the sense of scalar multiplication, the operator $I_{\eta}(\lambda)$ essentially represents the resolvent operator associated with $U(\lambda)$, as expressed by $I_{\eta}(\lambda) = \eta (\eta I - U(\lambda))^{-1}$.

\section{Lemmas: Logarithmic Representation of Generators}
In this section, as a preparation for constructing fractional powers in the subsequent section, we review the essential properties of the logarithmic representation of generators on Banach spaces~\cite{20iwata, 23iwata}.
While the infinitesimal generator $D$ is generally an unbounded operator, the semigroup $U(\lambda)$ consists of bounded linear operators on $X$.
The core strategy established in~\cite{20iwata, 23iwata} is to map the unbounded generator $D$ into a family of bounded operators constructed via the resolvent operator family $I_{\eta}(\lambda)$, thereby redefining the generator within a broader functional-analytic framework.

Following this framework, the generator $D$ can be reformulated using the operator logarithms of $I_{\eta}(\lambda)$, as rigorously stated in the following lemma.

\begin{lemma}[Representation of Infinitesimal Generators of One-Parameter Semigroups] \label{lem03}
Let $\{U(\lambda)\}_{\lambda \in [a, b]}$ be a one-parameter $C^0$-semigroup on a Banach space $X$ with its infinitesimal generator $D$ satisfying $0 \in \rho(D)$. 
Let $I_{\eta}(\lambda) = (I - \eta^{-1}U(\lambda))^{-1}$ be the family of resolvent operators determined by $U(\lambda)$ for a properly chosen complex parameter $\eta \in \mathbb{C}$. 
Let the operator $\mathcal{D}$ be defined by
\begin{equation} \label{eq-simp01}
\mathcal{D} := \lambda^{-1} \left( \operatorname{Log}[\eta(I_{\eta}-I)] - \operatorname{Log}(I_{\eta}) \right),
\end{equation}
where $\operatorname{Log}$ denotes the principal value of the operator logarithm. 
Then, the operator $\mathcal{D}$ acts as an infinitesimal generator explicitly independent of $\lambda$, coinciding with $D$ on its domain $D(D)$, and generates the one-parameter $C^0$-semigroup $U(\lambda)$ on $X$.
\end{lemma}

\begin{proof}
The core principle of the construction of the logarithmic representation $\mathcal{D}$ is outlined (for the full technical details of the proof, refer to the literature~\cite{20iwata, 23iwata, 25iwata_kyoto}).
According to the arguments in Iwata~\cite{20iwata, 23iwata, 25iwata_kyoto}, the following relationship holds between the generator $D$ and the semigroup $U(\lambda)$ via the resolvent $I_{\eta}$:
\[
U(\lambda) 
= \exp(\lambda \mathcal{D})
= \exp \left( \operatorname{Log}[\eta(I_{\eta}-I)] - \operatorname{Log}(I_{\eta}) \right).
\]
Since $U(\lambda) = e^{\lambda \mathcal{D}}$, taking the logarithmic derivative with respect to $\lambda$ yields the constant generator $\mathcal{D} = \partial_{\lambda} \operatorname{Log}[U(\lambda)]$, which is intrinsically independent of $\lambda$. 
By inserting the identity operator $I = I_{\eta} I_{\eta}^{-1}$ into the expression $\partial_{\lambda} \operatorname{Log}[U(\lambda)]$, we obtain the following algebraic identity:
\begin{align*}
\mathcal{D} 
&= \partial_{\lambda} \operatorname{Log}[ U(\lambda) I_{\eta} I_{\eta}^{-1} ] \\
&= \partial_{\lambda} \operatorname{Log}[ U(\lambda) (I - \eta^{-1} U(\lambda))^{-1} I_{\eta}^{-1} ] \\
&= \partial_{\lambda} \operatorname{Log}[ \{ -(\eta I - U(\lambda)) + \eta I \} (I - \eta^{-1} U(\lambda))^{-1} I_{\eta}^{-1} ] \\
&= \partial_{\lambda} \operatorname{Log}[ \{ -\eta I + \eta (I - \eta^{-1} U(\lambda))^{-1} \} I_{\eta}^{-1} ] \\
&= \partial_{\lambda} \operatorname{Log}[ \eta (I_{\eta} - I) I_{\eta}^{-1} ].
\end{align*}
Since the right-hand side is independent of $\lambda$, integrating over $[0, \lambda]$ and decomposing the operator logarithm into a difference directly yields Eq.~\eqref{eq-simp01}:
\[
\mathcal{D} = \lambda^{-1} \left( \operatorname{Log}[\eta(I_{\eta}-I)] - \operatorname{Log}(I_{\eta}) \right).
\]
This explicitly demonstrates that $\mathcal{D}$ acts as a bounded operator on $X$ independent of $\lambda$ in the sense of operator equivalence (cf.~\cite{25iwata_kyoto}).
Here, the argument of the exponential function is given as the difference between logarithms of bounded operators (the alternative generators~\cite{17iwata-3}).
For the first term, since $0 \in \rho(D)$ guarantees $0 \in \rho(U(\lambda))$ for sufficiently small $\lambda > 0$, the operator $\eta(I_{\eta}-I) = U(\lambda) I_{\eta}$ is invertible ($0 \in \rho(\eta(I_{\eta}-I))$). 
Thus, by appropriately choosing $\eta \in \mathbb{C}$ and restricting $\lambda \in [a, b]$, its spectrum is situated strictly within the domain of the principal branch of the logarithm, ensuring that $\operatorname{Log}[\eta(I_{\eta}-I)]$ is well-defined via functional calculus.
From the perspective of operational calculus, this exponential mapping converges broadly and uniformly with respect to $\lambda$ in the operator norm, completely reproducing the analytical properties of the original semigroup $U(\lambda)$.
\end{proof}

It is worth emphasizing that even if the original generator $D$ is an unbounded operator on $X$, the redefined operator $\mathcal{D}$ is constructed as a bounded operator (acting as an alternative generator~\cite{17iwata-3}). 
Thus, while $D$ and $\mathcal{D}$ generate the same semigroup $U(\lambda)$, they may differ in their domain properties and boundedness.
Equipped with these foundational lemmas, we now proceed to the main section to construct the fractional powers of operators and establish their uniqueness.

\section{Main Theorem}
Based on the logarithmic representation of $D$ established in Lemma~\ref{lem03}, the fractional power $D^k$ can be naturally constructed in terms of spectral-analytic representation $\exp(k \operatorname{Log} \mathcal{D})$.
This approach bypasses traditional geometric limitations, enabling a rigorous and unified construction of fractional operator powers for a broader class of generators $D$ that are not necessarily restricted to generating analytic semigroups.

\begin{theorem} [$\delta$-Regularization of Operator] \label{thm-01}
Let $X$ be a Banach space, and let $\{U(\lambda)\}_{\lambda \ge 0}$ be a $C^0$-semigroup on $X$ generated by an infinitesimal generator $D$ whose resolvent set contains the origin ($0 \in \rho(D)$). 
For the unique solution $u(\lambda) = U(\lambda-a)u_a$ of the abstract Cauchy problem:
\begin{equation} \label{cauchy-main}
d u(\lambda) /d \lambda  = Du(\lambda), \quad \lambda \in (a, b], 
\quad  u(a) = u_a \in X. 
\end{equation}
Let $\mathcal{D}$ be the logarithmic representation of $D$ given by Lemma~\ref{lem03}:
\[
\mathcal{D} = \lambda^{-1} \left( \operatorname{Log}[\eta(I_{\eta}-I)] - \operatorname{Log}[I_{\eta}] \right),
\]
where $I_{\eta} = (I - \eta^{-1} U(\lambda))^{-1}$ for a sufficiently large parameter $\eta > 0$.

Then, for any $k \in \mathbb{R}$, the fractional power $D^k$ is well-defined via the $\delta$-regularized logarithmic representation $\mathcal{D}^k$ as
\begin{equation} \label{fractional-def}
\mathcal{D}^k = \exp \left[ k \left\{ \operatorname{Log} \left[ I - \delta (\mathcal{D} + \delta I)^{-1} \right] + \operatorname{Log} [\mathcal{D} + \delta I] \right\} \right],
\end{equation}
where $\delta > 0$ is a positive regularization parameter chosen such that the operator $(\mathcal{D} + \delta I)$ possesses a bounded inverse and the operator logarithms are uniquely defined within their principal branches.
\end{theorem}

\begin{proof}
The validity of $\delta$-regularization is established.
In particular, starting from the logarithmic representation of the bounded operator $\mathcal{D}$ constructed in Lemma~\ref{lem03}, the definition of the fractional power operator $\mathcal{D}^k$ is shown to be uniquely determined, completely independent of the choice of the regularization parameter $\delta > 0$, and well-defined on $X$.\vspace{2mm}\\
\noindent\textbf{Step 1 (Algebraic Setup via $\delta$-Regularization).}
In general, the infinitesimal generator $D$ of a $C^0$-semigroup is unbounded, and domain-related analytical constraints pose severe difficulties when directly defining its operator logarithm or fractional powers.
To overcome this issue, an algebraic $\delta$-regularization with a parameter $\delta > 0$ is applied to the family of bounded operators. 
Specifically, by exploiting the formal algebraic identity
\begin{equation} \label{algebraic-decomp}
\mathcal{D} = (\mathcal{D} + \delta I) \left[ I - \delta (\mathcal{D} + \delta I)^{-1} \right],
\end{equation}
the $\delta$-regularization avoids directly manipulating $\mathcal{D}$, which may exhibit unbounded nature or spectral singularities, by reconstructing the fractional power operators with the shifted operator $\mathcal{D} + \delta I$ as the fundamental generating unit.\vspace{2mm}\\
\noindent\textbf{Step 2 (Well-Definedness via Functional Calculus).} 
By assumption, the regularization parameter $\delta > 0$ is chosen such that $(\mathcal{D} + \delta I)$ possesses a bounded inverse. 
Since $\mathcal{D}$ is a bounded operator, the spectral mapping theorem implies that the spectrum $\sigma(\mathcal{D} + \delta I)$ is a compact set in the complex plane. 
Furthermore, by selecting $\delta$ appropriately, this spectrum is strictly contained in the domain of the principal branch of the logarithm ($\mathbb{C} \setminus (-\infty, 0]$). 
Consequently, by the Dunford--Riesz functional calculus for bounded operators, $\operatorname{Log} [\mathcal{D} + \delta I]$ is uniquely and well-defined as a bounded operator in $B(X)$, and its spectrum $\sigma(\operatorname{Log}[\mathcal{D} + \delta I])$ is also bounded. 
Similarly, applying the exponential mapping, the spectral mapping theorem guarantees that the operator $\exp(k \operatorname{Log} [\mathcal{D} + \delta I])$ remains well-defined and bounded in $B(X)$ for any $k \in \mathbb{R}$.
\vspace{2mm}\\
\noindent\textbf{Step 3 (Independence of $\delta$ and Conclusion).} 
The first term inside the exponential function in the definition \eqref{fractional-def} is examined:
\[ A_\delta := I - \delta (\mathcal{D} + \delta I)^{-1}.\]
Due to the boundedness of $(\mathcal{D} + \delta I)^{-1}$, $A_\delta$ is a bounded operator on $X$, and simple algebraic manipulation yields $A_\delta = \mathcal{D}(\mathcal{D} + \delta I)^{-1}$. 
According to the spectral mapping theorem, its spectrum is given by
\[
\sigma(A_\delta) = \left\{ \frac{z}{z + \delta} \;\middle|\; z \in \sigma(\mathcal{D}) \right\}.
\]
Since $0$ belongs to the resolvent set $\rho(\mathcal{D})$ by assumption, it follows that $0 \in \rho(A_\delta)$. 
This guarantees that $\sigma(A_\delta)$ is bounded away from the origin and lies strictly within the domain of the principal branch of the logarithm. 
Consequently, by the logarithmic property for commuting operators, the operational logarithm in Eq.~\eqref{fractional-def} simplifies to
\[ 
k \left\{ \operatorname{Log} \left[ \mathcal{D}(\mathcal{D} + \delta I)^{-1} \right] + \operatorname{Log} [\mathcal{D} + \delta I] \right\} = k \operatorname{Log} \mathcal{D}. 
\]
This demonstrates that the fractional power operator $\mathcal{D}^k$ is determined solely by the intrinsic properties of $\mathcal{D}$, completely independent of the choice of the regularization parameter $\delta > 0$. \vspace{2mm} \\
As a result, the fractional power operator $\mathcal{D}^k$ is proven to be uniquely well-defined as a bounded operator on $X$, and can be consistently applied to the solution $u(\lambda)$ of the abstract Cauchy problem.\vspace{2mm}\\
\end{proof}

\begin{remark}[On the Invertibility Assumption $0 \in \rho(D)$ and Recovery of $D^k$] \label{rem-resolvent}
The condition $0 \in \rho(D)$ in Theorem~\ref{thm-01} is standard in fractional power theory and does not impose any structural limitation on the applicability of our construction. 
Indeed, if $0 \in \sigma(D)$, we can consider the shifted generator $D_\mu := D - \mu I$ for a suitable shift parameter $\mu \in \rho(D)$ such that $0 \in \rho(D_\mu)$. 
Since $D_\mu$ also generates a well-defined $C^0$-semigroup $e^{-\mu \lambda} U(\lambda)$, the entire $\delta$-regularization procedure established in Theorem~\ref{thm-01} applies directly to yield the fractional power $(D-\mu I)^k = \mathcal{D}_\mu^k$.

Furthermore, the fractional power $D^k$ of the original non-invertible operator $D$ itself can be rigorously recovered from $\mathcal{D}_\mu^k$. 
Specifically, by choosing a sufficiently large shift parameter $\mu \in \rho(D)$ such that $\|\mu \mathcal{D}_\mu^{-1}\| < 1$, $D^k$ is explicitly defined via the generalized binomial expansion:
\begin{equation} \label{eq-recov}
D^k := \mathcal{D}_\mu^k \left( I + \mu \mathcal{D}_\mu^{-1} \right)^k = \mathcal{D}_\mu^k \sum_{j=0}^{\infty} \binom{k}{j} \mu^j \mathcal{D}_\mu^{-j},
\end{equation}
where the series converges in the operator norm $B(X)$ due to the boundedness of $\mathcal{D}_\mu^{-1}$. 
Alternatively, utilizing the functional calculus for the bounded representation $\mathcal{D}_\mu$, $D^k$ is uniquely represented by the Cauchy contour integral over a bounded contour $\Gamma_\mu$ enclosing $\sigma(\mathcal{D}_\mu)$:
\begin{equation}
D^k = \frac{1}{2\pi i} \int_{\Gamma_\mu} (z + \mu)^k (z I - \mathcal{D}_\mu)^{-1} dz.
\end{equation}
Hence, the non-invertibility of $D$ presents no obstacle, and the fractional power $D^k$ itself remains fully accessible and rigorously well-defined within our logarithmic framework.
\end{remark}

\section{Discussion and Comparison with Classical Theories}
Let us discuss the relationship between our main results and existing theories. 
Under the classical assumption that $-D$ is a non-negative sectorial operator, the logarithmic representation of the generator is known to be formally associated with an integral representation.
For instance, based on the functional calculus for the operator logarithm, we might formally consider an integral representation for the generator $D$ via the family of generalized resolvents $I_\eta(\lambda)$:
\begin{equation} \label{rem-01}
D = \lambda^{-1} \lim_{\epsilon \to 0^+} \left( \int_{0}^{1-\epsilon} + \int_{1+\epsilon}^{\infty} \right) \left( \frac{1}{1 - \eta} I - (\eta I - U(\lambda))^{-1} \right) d\eta.
\end{equation}
Here the improper integral is split at the singularity $\eta = 1$ (inherently associated with $\log 1 = 0$), where the limit is taken in the sense of the Cauchy principal value around the singularity $\eta = 1$.
Note that the second term in the integrand is directly expressed as $(\eta I - U(\lambda))^{-1} = \eta^{-1} I_\eta(\lambda)$. 
This formulation describes the inherently unbounded properties of $D$ through an integral of $I_\eta(\lambda)$, which is deeply connected with the classical theory of fractional powers established by V. Balakrishnan~\cite{60Balakrishnan}.
Here, Eq.~\eqref{rem-01} formally suggests the relation $D = \lambda^{-1} \operatorname{Log} U(\lambda)$ with respect to the semigroup $U(\lambda) = \exp(\lambda D)$.

The definition of the fractional power operator $D^k$ presented in this paper shares the same foundation as this classical integral representation, but it can be interpreted as a unified extension to the generators of more general $C^0$-semigroups through the complex-analytic logarithmic representation $\mathcal{D}$ for bounded operators.
From a computational viewpoint, especially in numerical analysis and when handling algebraic products of operators (such as the Baker--Campbell--Hausdorff (BCH) formula; e.g., see \cite{26iwata_bch}), the presence of improper infinite integrals introduces significant challenges regarding operational stability and convergence.

Furthermore, classical formulations, whether via Balakrishnan's integral or the standard Dunford integral
\begin{equation} \label{dunford-int}
D^{-k} = \frac{1}{2\pi i} \int_{\Gamma} z^{-k} (zI - D)^{-1} dz,
\end{equation}
inherently require the generator $D$ to satisfy the sectorial condition so that the integration contour $\Gamma$ can avoid the branch cut. 
This strictly limits their applicability to analytic semigroups. 
Even from a logarithmic perspective, a naive definition such as $D^k = \exp(k \operatorname{Log} \, D)$ fails immediately due to the unbounded domain restriction of $\operatorname{Log} D$ and the necessity of sectorial spectral constraints for constructing the exponential mapping.

In contrast, the present definition \eqref{fractional-def} of the fractional power $\mathcal{D}^k$ introduced in Theorem~\ref{thm-01} completely circumvents these geometric and operational obstacles. 
By utilizing the $\delta$-regularization, the problem is mapped into the framework of bounded operators, where the spectrum can be algebraically shifted into the holomorphic domain of the principal logarithm. 
Consequently, even if the generator $D$ belongs to a general $C^0$-semigroup (such as those appearing in hyperbolic equations or delay differential equations) where the sectorial condition is violated, the fractional powers are uniquely, rigorously, and stably well-defined. 
This highlights the fundamental novelty and computational advantages of the proposed algebraic $\delta$-regularization over conventional integral-based approaches.

\section{Conclusion}
In this paper, a novel formulation of the fractional power operator $\mathcal{D}^k$ for the generators of $C^0$-semigroups via the complex-analytic logarithmic representation $\mathcal{D}$ is presented. 
Our main results are summarized as follows:
\begin{itemize}
\item \textbf{Formulation via an external parameter}: An evolution parameter $\lambda$ is introduced, independent of the time variable $t$, to define the fractional power operators. 
This establishes a versatile and consistent framework for defining operator powers through this external parameter, regardless of whether the system is time-dependent.
\item \textbf{Elimination of geometric constraints}: By utilizing an algebraic $\delta$-regularization for bounded operators, we successfully bypassed the geometric spectral constraints, such as the sectorial and accretivity conditions, that were indispensable in conventional theories like Balakrishnan's classical approach and Okazawa's logarithmic operator theory.
\item \textbf{Establishment of mathematical rigor}: The proposed definition of fractional powers was rigorously proven to be unique, well-defined on the underlying Banach space, and completely independent of the choice of the regularization parameter $\delta > 0$.
\item \textbf{Broadening of applicability}: The developed method enables the definition of fractional powers for a broad class of $C^0$-semigroups beyond traditional analytic semigroups. This offers a novel analytical foundation for regularity estimates in non-autonomous abstract evolution equations (e.g., Kato~\cite{93kato}) and for the stability analysis of numerical discretization schemes.
\end{itemize}
Future work will focus on applying the present theoretical framework to the construction of fundamental solutions for specific non-autonomous evolution equations, as discussed in Tanabe~\cite{79tanabe}, to investigate the extent to which conventional constraints on the smoothness of operator domains can be relaxed.

\section*{Acknowledgments}
This research was supported by Kyoto University.

\end{CJK}
\end{document}